\documentclass{icmart}
\usepackage{amssymb}
\usepackage[pagebackref,colorlinks,linkcolor=red,citecolor=blue,urlcolor=blue,hypertexnames=false]{hyperref}
\usepackage{amsrefs}
\usepackage{amscd}
\usepackage{undertilde}
\usepackage{url}
\input xypic
\xyoption{all}

\contact[gcorti@dm.uba.ar]{Dep. Matem\'atica-IMAS, FCEyN-UBA\\ Ciudad Universitaria Pab 1\\
C1428EGA Buenos Aires, Argentina}

\newdir{ >}{{}*!/-5pt/@{>}}

\newcommand{\topdf}{\texorpdfstring}

\newcommand{\SchF}{\mathrm{Sch}/k}

\newcommand{\cdh}{\mathrm{cdh}}
\def\Spec{\operatorname{Spec}}

\def\cB{\mathcal B}

\def\TC{\mathcal{TC}}

\def\fp{\mathfrak p}
\def\fc{\mathfrak c}
\newcommand{\locass}{\mathfrak{LocAlg}}

\def\fC{\mathfrak{C}}

\def\cF{\mathcal F}

\def\cG{\mathcal G}

\def\cS{\mathcal S}
\def\cK{\mathcal K}

\def\ring{\mathrm{Rings}}
\def\rings{\mathrm{Rings}}

\def\tr{\mathrm{tr}}
\def\spt{\mathrm{Spt}}

\def\ab{\mathfrak Ab}

\def\red{\mathrm{red}}
\newcommand\topo{\mathrm{top}}

\def\inf{\mathrm{inf}}
\def\ninf{\mathrm{ninf}}
\def\nil{\mathrm{nil}}

\def\hotimes{\hat{\otimes}}
\def\sotimes{\utilde{\otimes}}

\newcommand{\zA}{\mathbb{A}}
\newcommand{\zC}{\mathbb{C}}
\newcommand{\zH}{\mathbb{H}}

\newcommand{\zQ}{\mathbb{Q}}

\newcommand{\zZ}{\mathbb{Z}}
\newcommand{\zN}{\mathbb{N}}

\newcommand{\hofi}{\mathrm{hofiber}}
\newcommand{\fA}{\mathfrak{A}}

\def\coker{\operatorname{coker}}
\def\colim{\operatornamewithlimits{colim}}

\def\Spec{\operatorname{Spec}}

\def\lra{\longrightarrow}

\def\iso{\stackrel{\cong}\lra}
\def\triqui{\vartriangleleft}
\def\weq{\overset\sim\lra}

\numberwithin{equation}{section}

\newtheorem{theorem}{Theorem}[section]
\newtheorem{corollary}[theorem]{Corollary}

\newtheorem{proposition}[theorem]{Proposition}
\newtheorem*{conjecture}{Conjecture}
\newtheorem*{question}{Question}

\theoremstyle{definition}

\newtheorem{remark}[theorem]{Remark}

\title[Excision, descent, and singularity in algebraic $K$-theory]{Excision, descent, and singularity in algebraic $K$-theory}

\author[Guillermo Corti\~nas]
{Guillermo Corti\~nas \thanks{Supported by CONICET and partially supported by grants UBACyT 20020100100386, PIP 11220110100800 and MTM2012-36917-C03-02.}}

\begin{document}

\begin{abstract}
Algebraic $K$-theory is a homology theory that behaves very well on sufficiently nice objects such as stable $C^*$-algebras or smooth algebraic varieties, and very badly in singular situations. This survey explains how to exploit this to detect singularity phenomena using 
$K$-theory and cyclic homology.
\end{abstract}



\maketitle
\section{Introduction}\label{sec:intro}
A homology theory of rings associates groups $H_n(R)$ ($n\in\zZ$) depending covariantly on the ring $R$. We say that 
$H$ is (polynomially) \emph{homotopy invariant} if the map $H_n(R)\to H_n(R[t])$ is an isomorphism. A homology theory also associates groups $H_n(R:I)$ to every two-sided ideal $I\triqui R$, which fit into a long exact sequence
\[
H_{n+1}(R/I)\to H_n(R:I)\to H_n(R)\to H_n(R/I).
\]
We say that $H$ is \emph{nilinvariant} if $H_*(R:I)=0$ when $I$ is nilpotent. The homology theory is said to satisfy \emph{excision} if whenever $f:R\to S$ is a ring homomorphism sending an ideal $I\triqui R$ isomorphically onto an ideal of $S$, the map $H_*(R:I)\to H_*(S:f(I))$ is an isomorphism. Similarly, a cohomology theory of algebraic varieties over a field (schemes of finite type) associates groups $H_n(X:Y)$ with every closed immersion $Y\to X$, depending contravariantly on $(X:Y)$. In particular if $\tilde{X}$ is the blowup of $X$ along $Y$ and $\tilde{Y}$ is the exceptional fiber, we have a map $H_*(X:Y)\to H_*(\tilde{X}:\tilde{Y})$; this map is an isomorphism whenever $H$ satisfies \emph{$cdh$-descent}. For example, Quillen's \emph{algebraic $K$-theory} $K$ is a (co)homology theory of rings and schemes which has none of the aforementioned properties. For another example, Weibel's \emph{homotopy algebraic $K$-theory} $KH$ is also defined for rings and schemes, and satisfies all of them. There is a natural map $K_n(R)\to KH_n(R)$ which is an isomorphism when $R$ is \emph{$K_n$-regular}; this means that the map $K_n(R)\to K_n(R[t_1,\dots,t_p])$ is an isomorphism for all $p$. For instance Noetherian regular rings such as rings of polynomial functions on smooth varieties, and stable $C^*$-algebras such as the algebra $\cK$ of compact operators, are $K_n$-regular for all $n$.  In general the map $K\to KH$ is part of a long exact sequence
\[
KH_{n+1}(R)\to K_n^{\nil}(R)\to K_n(R)\to KH_n(R).
\]
In this article we survey a series of results that help describe the groups $K_*^{\nil}$ in terms of cyclic homology, and explain how this description has helped make significant progress in several long standing problems, ranging from the comparison of algebraic and topological $K$-theory of topological algebras to the relation between $K$-regularity and nonsingularity of algebraic varieties. 

The article is organized as follows. In Section \ref{sec:obse} we recall (from Goodwillie's paper \cite{goo2} and from \cite{kabi}) two key properties of the Chern character $ch_*:K_*(R)\to HN_*(R)$  to negative cyclic homology of $\zQ$-algebras. They can be summarized by saying that \emph{infinitesimal $K$-theory} is nilinvariant and satisfies excision (Theorem \ref{thm:kinf}); the infinitesimal $K$-groups fit into a long exact sequence
\[
HN_{n+1}(R)\to K^{\inf}_n(R)\to K_n(R)\overset{ch_n}{\longrightarrow} HN_n(R).
\]
In Section \ref{sec:kinfreg} we explain how the results of the previous one were used in \cite{CT} to compare the algebraic and the topological $K$-theory of a stable locally convex algebra $L$. The main result reviewed in this section is Theorem \ref{thm:freche}, which says that for such $L$, $KH_*(L)=K_*^{\topo}(L)$ and $K^{\nil}_n(L)=HC_{n-1}(L)$ is cyclic homology. In Section \ref{sec:mv} we recall the notions of descent and Mayer-Vietories properties. We review Thomason's theorems that $K$-theory satisfies Nisnevich descent and has the Mayer-Vietoris property for blow-ups along regularly embedded closed subschemes (Theorem \ref{thm:bt}), and their analogues for cyclic homology and infinitesimal $K$-theory of schemes of finite type over a field of characteristic zero (Theorem \ref{thm:reghc}). In Section \ref{sec:chh} we review Haesemeyer's theorem on $cdh$-descent and its generalization (Theorems \ref{thm:chh} and \ref{thm:chhg}) and derive from them a long exact sequence (Theorem \ref{thm:fhc})
\begin{equation}\label{intro:fhc}
KH_{n+1}(X)\to \cF^{HC}_{n-1}(X)\to K_n(X)\to KH_n(X),
\end{equation}
for every scheme $X$ of finite type over a field of characteristic zero. Up to extension, the groups $\cF_*^{HC}(X)$ are computed from the cyclic homology groups of $X$ and of an array of smooth schemes that appear in its desingularization process. In Section \ref{sec:dimconj} we show how these results were used to prove Weibel's dimension conjecture for schemes of finite type over a field of characteristic zero (Theorem \ref{thm:dim}). In Section \ref{sec:reg} we begin by explaining how the sequence \ref{intro:fhc} can be used to compute the obstruction to $K_n$-regularity (Proposition \ref{prop:obsreg}). Then we review several results on $K$-regularity, including that Vorst's regularity conjecture and Gubeladze's nilpotence conjecture hold for algebras and coefficient rings containing a field of characteristic zero (Theorems \ref{thm:chwvorst} and \ref{thm:gubel0}), and the answer to Bass' question on whether $K_n(R)=K_n(R[t])$ implies that $K_n(R)=K_n(R[t_1,t_2])$ 
(Theorem \ref{thm:bass}). Versions of some of these results in characteristic $p>0$ are reviewed in Section \ref{sec:p}. They are based on the good properties of the B\"okstedt-Hsiang-Madsen cyclotomic trace $\tr: K\to TC$, which takes values in topological cyclic homology \cite{bhm}. Theorems of McCarthy and Geisser-Hesselholt ( \ref{thm:mc} and \ref{thm:ghkabi}) say that $\tr$ computes the obstruction groups to nilinvariance and excision up to $p$-adic completion. The results of Geisser-Hesselholt extending Haesemeyer's theorem (Theorem \ref{thm:ghchh}) and proving Weibel's dimension conjecture (Theorem \ref{thm:ghneg}) and Vorst's regularity conjecture (Theorem \ref{thm:ghvorst}) over a perfect field of characteristic $p>0$ which admits strong resolution of singularities are reviewed, as well as the solution of Gubeladze's nilpotence conjecture over a field of charactristic $p$ (Theorem \ref{thm:gubelp}).

\bigskip
\noindent{\em Acknowledgements}. I wish to thank the following people for helpful comments on previous versions of this article: Christian Haesemeyer, Lars Hesselholt, Emanuel Rodr\'\i guez Cirone, Andreas Thom and Bruce Williams.

\bigskip

\section{The obstructions to excision and nilinvariance}\label{sec:obse}
Quillen's algebraic $K$-theory associates groups $K_n(R)$ ($n\in\zZ$) to each associative, not necessarily unital ring $R$. These groups are defined as the stable homotopy groups of a functorial spectrum $K(R)$. If $I\triqui R$ is an ideal, the relative $K$-theory spectrum 
is $K(R:I)=\hofi(K(R)\to K(R/I))$. The spectrum $K(R:I)$ is defined so as to fit into a \emph{homotopy fibration sequence}
\[
K(R:I)\to K(R)\to K(R/I).
\]
Thus taking homotopy groups we obtain a long exact sequence
\[
K_{n+1}(R/I)\to K_n(R:I)\to K_n(R)\to K_n(R/I).
\]
Observe that the failure of the map $K_*(R)\to K_*(R/I)$ to be an isomorphism is measured by the relative groups $K_*(R:I)$. If $I$ is nilpotent, the groups $K_n(R:I)$ vanish for $n\le 0$. Thus $K_n(R)\to K_n(R/I)$ is an isomorphism for $n\le 0$; because of this, we say that $K$-theory is \emph{nilinvariant} in nonpositive degrees. For $n\ge 1$ the groups $K_n(R:I)$ may be nonzero for nilpotent $I$; if $R$ is a $\zQ$-algebra they are $\zQ$-vector spaces (\cite{chumodp}*{Consequence 1.4}). For general $R$, the groups $K_n(R:I)\otimes\zQ$ are computed by means of the Chern character \cite{goo2}
\begin{equation}\label{map:ch}
ch:K_*(R)\to HN_*(R\otimes\zQ).
\end{equation}
Negative cyclic homology is connected with cyclic and periodic cyclic homology by means of Connes' $SBI$-sequence
\begin{equation}\label{seq:sbi}
HP_{n+1}(R)\overset{S}\to HC_{n-1}(R)\overset{B}\to HN_n(R)\overset{I}\to HP_n(R).
\end{equation}

\begin{theorem}\label{thm:goo}(Goodwillie, \cite{goo1}*{Theorem II.5.1},\cite{goo2}*{Main Theorem})
Let $R$ be a ring and $I\triqui R$ a nilpotent ideal. Then 
\[
HP_n(R\otimes\zQ:I\otimes\zQ)=0,
\]
and the Chern character induces an isomorphism
\[
K_n(R:I)\otimes\zQ\cong HN_n(R\otimes\zQ:I\otimes\zQ)\cong HC_{n-1}(R\otimes\zQ:I\otimes\zQ).
\]
In particular, $K_n(R:I)\otimes\zQ\cong K_n(R\otimes\zQ:I\otimes\zQ)$.
\end{theorem} 
\begin{remark}\label{rem:goonital} Goodwillie states his results for unital $R$; the nonunital case follows from the unital case (see \cite{derived}*{Lemma 6.1}, e.g.). 
\end{remark}
If $I\triqui R$ is an ideal and $f:R\to S$ is a ring homomorphism such that $f(I)\triqui S$ is an ideal and
$f:I\to f(I)$ is bijective, we put 
\[
K(R,S:I)=\hofi(K(R:I)\to K(S:f(I))).
\] 
The groups $K_n(R,S:I)$ are zero for $n\le -1$. Thus for $n\le 0$ the groups $K_n(R:I)$ depend only on $I$ and not on the ideal embedding $I\triqui R$. Because of this, we say that $K$-theory satisfies \emph{excision} (or that it is \emph{excisive}) in nonpositive degrees. If $R$ is a $\zQ$-algebra, the groups
$K_n(R,S:I)$ are $\zQ$-vector spaces (\cite{chumodp}*{Consequence 1.5}). For general $R$, the groups $K_n(R,S:I)\otimes\zQ$ are again computed by means of the Chern character; this is the content of Theorem \ref{thm:kabi} below. First let us recall the Cuntz-Quillen excision theorem \cite{cq}*{Theorem 5.3}, which says that $HP$ \emph{satisfies excision} in the category of  $\zQ$-algebras; this means that if in the situation just described $f$ is a $\zQ$-algebra homomorphism, then
\begin{equation}\label{cq}
HP_*(R,S:I)=0.
\end{equation}
It follows that the $B$-map in the $SBI$ sequence \eqref{seq:sbi} induces an isomorphism
\[
HC_{*-1}(R,S:I)\cong HN_*(R,S:I).
\]

\begin{theorem}\label{thm:kabi}(\cite{kabi}*{Theorem 0.1})
Let $I\triqui R$ be an ideal and let $f:R\to S$ be a ring homomorphism. Assume that $f(I)\triqui S$ is an ideal and that $f:I\to f(I)$ is bijective. Then the Chern character induces an isomorphism
\begin{align*}
K_*(R,S:I)\otimes\zQ&\cong HN_*(R\otimes\zQ,S\otimes\zQ:I\otimes\zQ)\\
&\cong HC_{*-1}(R\otimes\zQ,S\otimes\zQ:I\otimes\zQ).
\end{align*}
In particular, $K_*(R,S:I)\otimes\zQ\cong K_*(R\otimes\zQ,S\otimes\zQ:I\otimes\zQ)$.
\end{theorem}

In the case of $\zQ$-algebras Theorem \ref{thm:kabi} confirmed a conjecture formulated in the mid ~1980's by  Geller-Reid-Weibel (\cite{grw}, \cite{gw}). Its proof combines the techniques developed by Cuntz-Quillen to prove their excision theorem with those used by Suslin-Wodzicki to characterize those nonunital rings $I$ such that $K_*(-,-:I)\otimes\zQ=0$ \cite{sw}*{Theorem A}.

In the applications considered in the next sections, all the rings involved are $\zQ$-algebras. In this case we will find it useful to formulate the results
above in terms of infinitesimal $K$-theory. The Chern character \eqref{map:ch} comes from a map of spectra $K\to HN$; \emph{infinitesimal $K$-theory} is the homotopy fiber of this map:
\[
K^{\inf}(R)=\hofi(K(R)\to HN(R)).
\]
In terms of $K^{\inf}$, Theorems \ref{thm:goo} and \ref{thm:kabi} can be expressed as follows. 

\begin{theorem}\label{thm:kinf}
Infinitesimal $K$-theory of $\zQ$-algebras is nilinvariant and satisfies excision. 
\end{theorem} 

\section{\topdf{$K^{\inf}$}{Kinf}-regularity and algebraic \textup{K}-theory of topological algebras.}\label{sec:kinfreg}

Let $A$ be a ring and $F:\rings\to\ab$ a functor. We say that $A$ is \emph{$F$-regular} if the map induced
by the canonical inclusion 
\[
F(A)\to F(A[t_1,\dots,t_n])
\]
is an isomorphism for all $n\ge 1$. The functor $F$ is called \emph{invariant under polynomial homotopy} (homotopy invariant, for short) on a subcategory $\fC\subset\rings$ if every ring $A\in\fC$ is $F$-regular. A \emph{homology theory} of rings is a covariant functor $E$ from rings to spectra; we write $E_n(R)=\pi_nE(R)$ for the stable homotopy group. If $E$ is a homology theory of rings, we call a ring \emph{$E$-regular} if it is $E_n$-regular for all $n$. We call $E$ \emph{homotopy invariant} if each $E_n$ is homotopy invariant. There is a well-known procedure $E\to EH$, the \emph{homotopization procedure} that transforms $E$ to a homotopy invariant homology theory $EH$. This construction comes with a natural map $E(A)\to EH(A)$ and setting
$E^{\nil}(A)=\hofi(E(A)\to EH(A))$ one obtains a long exact sequence
\begin{equation}\label{seq:nil}
EH_{n+1}(A)\to E_n^{\nil}(A)\to E_n(A)\to EH_n(A).
\end{equation}
If $A$ is $E_m$-regular for all $m\le n$, then $E_m(A)\to EH_m(A)$ is an isomorphism for $m\le n$ and is onto for $m=n+1$. Applying this procedure to Quillen's $K$-theory yields Weibel's homotopy $K$-theory $KH$ \cite{kh}. The following theorem summarizes two key properties of $KH$.

\begin{theorem}\label{thm:kh}(Weibel, \cite{kh}*{Theorems 2.1 and 2.3})
Homotopy algebraic $K$-theory $KH$ is nilinvariant and satisfies excision.
\end{theorem}

If $A$ is a $\zQ$-algebra, then $A$ is $HP$-regular (e.g. by \cite{goo1}*{Corollary II.4.4} or by \cite{kas}*{Equation 3.13}). Geller and Weibel showed in \cite{far}*{Theorem 4.1} that if $A$ is a $\zQ$-algebra then $HNH_*(A)\cong HP_*(A)$, and \eqref{seq:nil} identifies with Connes' $SBI$-sequence \eqref{seq:sbi}.
If $A$ is a $\zQ$-algebra, set 
\[
K^{\ninf}(A)=K^{\inf,\nil}(A).
\]
The homotopization procedure preserves fibration sequences. Hence we have a homotopy commutative diagram whose rows and columns are homotopy fibration sequences:
\begin{equation}\label{diag:knil}
\xymatrix{
K^{\ninf}(A)\ar[r]\ar[d]&K^{\inf}(A)\ar[d]\ar[r]& K^{\inf}H(A)\ar[d]\\
K^{\nil}(A)\ar[r]\ar[d]&K(A)\ar[d]\ar[r]&KH(A)\ar[d]\\
HC(A)[-1]\ar[r]&HN(A)\ar[r]&HP(A). 
}
\end{equation}
By the discussion above, 
\begin{equation}\label{kinfreghc}
A\ \ K^{\inf}-\mbox{regular }\Rightarrow K_*^{\nil}(A)\iso HC_{*-1}(A).
\end{equation}

A \emph{locally convex algebra} is a complete locally convex $\zC$-vector space $L$ equipped with a jointly continuous, associative multiplication $L\otimes_\zC L\to L$. The following theorem concerns the $K$-theory of a large class of stable locally convex algebras. It computes the obstruction for the map $K\to K^{\topo}$ to be an isomorphism. Here $K_*^{\topo}(L)=kk_*^{lc}(\zC,L)$ is the covariant part of Cuntz' bivariant $K$-theory for locally convex algebras \cite{cubiva}. As expected of a bivariant $K$-theory of complex topological algebras, it is periodic of period two:
\begin{equation}\label{bott}
kk_*^{lc}(L,M)=kk^{lc}_{*+2}(L,M). 
\end{equation}
Thus there are only two covariant topological $K$-groups, $K_0^{\topo}$ and $K_1^{\topo}$. 

Let $\cB=\cB(\ell^2(\zN))$ be the $C^*$-algebra of bounded operators in an infinite dimensional, separable Hilbert space. An ideal $I\triqui\cB$ is a \emph{Fr\'echet operator ideal} if it carries a complete metrizable
locally convex topology such that the inclusion $I\to\cB$ is continuous. 

\begin{theorem}\label{thm:freche} (\cite{CT}*{Theorems 6.2.1 and 6.3.1})
Let $L$ be a locally convex $\zC$-algebra and let $I$ be a Fr\'echet operator ideal. Write $\hotimes$ for the projective tensor product. Then:
\item[i)] $L\hotimes I$ is $K^{\inf}$-regular.
\item[ii)] $KH_*(L\hotimes I)=K_*^{\topo}(L\hotimes I)$.
\item[iii)] For each $n \in \zZ$,
there is a natural $6$-term exact sequence of abelian groups as follows:
\[
\xymatrix{ K^{\topo}_{1}(L\hotimes I)\ar[r]&HC_{2n-1}(L\hotimes I)\ar[r]&K_{2n}(L\hotimes I)\ar[d] \\
K_{2n-1}(L \hotimes I) \ar[u]& HC_{2n-2}(L \hotimes I) \ar[l] & K_0^{\topo}(L\hotimes I). \ar[l]}
\]
\end{theorem}

\noindent{\it Sketch of the proof of Theorem \ref{thm:freche}.} Call a Banach ideal $I\triqui\cB$ \emph{harmonic} if it contains an operator whose sequence of singular values is the harmonic sequence $\{1/n\}$. Results of Cuntz-Thom \cite{cut}*{Theorems 4.2.1 and 5.1.2} imply that if $H$ is a homology theory which satisfies excision and $I\triqui\cB$ is harmonic, then $H(-\hotimes I)$ is invariant under $C^\infty$ homotopies, also called \emph{diffotopies}. It is shown in \cite{CT}*{Theorem 6.16} that if in addition $H$ is nilinvariant, then $H(-\hotimes I)$ is diffotopy invariant for any Fr\'echet ideal $I$. In particular this applies to $K^{\inf}$ and $KH$ by Theorems \ref{thm:kinf} and \ref{thm:kh}. Part i) of the theorem follows from the fact that $K^{\inf}(-\hotimes I)$ is diffotopy invariant, using the argument of \cite{comparos}*{Theorem 3.4}. By \cite{CT}*{Lemma 3.2.1}, $K^{\inf}$-regular $\zQ$-algebras are $K_n$-regular for $n\le 0$. Hence we have $KH_n(L\hotimes I)=K_n(L\hotimes I)$ for $n\le 0$. Cuntz and Thom proved in \cite{cut}*{Theorem 6.2.1} that $K_0(L\hotimes I)=K_0^{\topo}(L\hotimes I)$ when $I$ is harmonic; by the argument of \cite{CT}*{Theorem 6.16}, this extends to all Fr\'echet ideals. Summing up $KH(-\hotimes I)$ is an excisive and diffotopy invariant homology theory in $\locass$ which agrees with $K^{\topo}(-\hotimes I)$ in dimension $0$. 
A standard argument now shows that they must agree in all dimensions (\cite{cudocu}*{Abschnitt 6}). This proves ii). Part iii) follows from what we have already done, using \eqref{kinfreghc}, \eqref{bott}, and diagram \eqref{diag:knil}.\qed

\bigskip

The following theorem can be proved using Theorem \ref{thm:freche}.

\begin{theorem}\label{thm:ubau}(\cite{CT}*{Theorem 8.3.3})
Let $\cK\triqui\cB$ be the ideal of compact operators and let $L$ be a Fr\'echet algebra whose topology is generated by a countable family of submultiplicative seminorms and which has a uniformly bounded, one-sided approximate unit. Then the map 
\[
K_*(L\hotimes\cK)\to K_*^{\topo}(L\hotimes\cK)
\]
is an isomorphism.
\end{theorem}

The particular case of the theorem above when $L$ is a Banach algebra with a one-sided approximate unit confirms a conjecture formulated by Karoubi in \cite{karcomp}.  Wodzicki announced a proof of the latter case in \cite{wodk}*{Theorem 1}; he told us he has also proved the general case of Theorem \ref{thm:ubau}. 
The proof of Theorem \ref{thm:ubau} given in \cite{CT} consists of showing that $HC_*(L\hotimes\cK)=0$, and then using Theorem \ref{thm:freche}. For a different proof see \cite{friendly}*{Theorem 12.1.1}. Karoubi also proved that if $\fA$ is a $C^*$-algebra
and $\sotimes$ is the spatial tensor product, then the comparison map 
\begin{equation}\label{map:karc}
K_*(\fA\sotimes\cK)\to K_*^{\topo}(\fA\sotimes\cK)
\end{equation}
is an isomorphism for $*\le 0$, and conjectured that it is also an isomorphism for $*\ge 1$. This conjecture was proved by Suslin-Wodzicki (\cite{sw}*{Theorem 10.9}). We remark that the analog of Theorem \ref{thm:freche} for $\fA\sotimes\cK$ also holds. Indeed a theorem of Higson (\cite{comparos}*{Theorem 20}) says that $\fA\sotimes\cK$ is $K$-regular, so $KH_*(\fA\sotimes\cK)=K_*^{\topo}(\fA\sotimes\cK)$, since  \eqref{map:karc} is an isomorphism; a similar argument as that of \cite{comparos}*{Theorem 20} shows that $\fA\sotimes\cK$ is also $K^{\inf}$-regular. Finally it is not hard to see that $HC_*(\fA\sotimes\cK)=0$ (the argument is similar to that of \cite{CT}*{Lemma 8.2.3}).

\begin{remark}\label{rem:nofunca}
No nonzero commutative unital $\zQ$-algebra is $K^{\inf}$-regular. Indeed \cite{CT}*{Proposition 3.3.1} says that if a unital algebra $R$ is $K^{\inf}_n$-regular for some $n\ge 1$, then every element $a\in R$ can be written as a finite linear combination
\[
a=\sum_{i=1}^m b_i(x_iy_i-y_ix_i)c_i \qquad (b_i,c_i,x_i, y_i\in R).
\]
\end{remark}

\section{Mayer-Vietoris sequences and descent}\label{sec:mv}

Fix a field $k$. Let $\SchF$ be the category of schemes essentially of finite type over $\Spec(k)$. A \emph{$k$-scheme} is an object of $\SchF$.  A \emph{cohomology theory} (also called \emph{presheaf of spectra}) is a functor $(\SchF)^{op}\to\spt$. Let 
\begin{equation}\label{diag:square}
\xymatrix{
Y'\ar[r]\ar[d]& X'\ar[d]^p\\
Y\ar[r]_i& X
}
\end{equation}
be a cartesian square in $\SchF$. A cohomology theory $E$ of $k$-schemes is said to have the \emph{Mayer-Vietoris} property with respect to \eqref{diag:square} if it sends \eqref{diag:square} to a homotopy cartesian diagram of spectra. This implies that we have a long exact sequence of Mayer-Vietoris type:
\[
E_{n+1}(Y')\to E_n(X)\to E_n(Y)\oplus E_n(X')\to E_n(Y').
\]
In the sequence above, we have used subscript notation; we will switch to superscripts when needed, following the usual convention $E_*=E^{-*}$. 

We consider classes of cartesian squares in $\SchF$ which are closed under isomorphism. We call such a class a \emph{$cd$-structure}. Each such class generates a Grothendieck topology on $\SchF$. We shall consider the following $cd$-structures; note that each of them is contained in the next one:
\begin{itemize}

\item The $cd$-structure of elementary \emph{Zariski} squares. A square \eqref{diag:square} is an elementary Zariski square if $i$ and $p$ are open immersions and $X=Y'\cup X'$. We write \emph{zar} for the Grothendieck topology generated by this structure. 

\item The $cd$-structure of elementary \emph{Nisnevich} squares. A square \eqref{diag:square} is an elementary Nisnevich square if $i$ is an open embedding, $p$ is \'etale, and $p:(X'\setminus Y')_\red\to (X\setminus Y)_\red$ is an isomorphism.  Here $X_\red$ is the reduced scheme. We remark that in the particular case when $p$ is an open immersion, the latter condition is equivalent to the condition that $Y'\cup X'=X$; thus every elementary Zariski square is an elementary Nisnevich square. We write \emph{nis} for the Nisnevich topology.

\item Voevodsky's \emph{combined $cd$-structure}. It consists of the elementary Nisnevich squares and the \emph{abstract blow-up squares}. 
A square \eqref{diag:square} is an abstract blow-up if $p$ is proper, $i$ is a closed embedding and $p:(X'\setminus Y')\to (X\setminus Y)$ is an isomorphism. We write \emph{cdh} for the corresponding topology. 
\end{itemize}

Let $c$ be one of the above $cd$-structures on $\SchF$ and let $t$ be the Grothendieck topology it generates. A cohomology theory $E$ of $k$-schemes satisfies \emph{descent} with respect to $t$ if it has the Mayer-Vietoris property for every square in $c$. There is a construction $(E,t)\mapsto \zH_t(-,E)$ which given a topology $t$ and a cohomology theory $E$ produces a cohomology theory $\zH_t(-,E)$ which satisfies $t$-descent, and a natural transformation (\cite{jardine})
\begin{equation}\label{map:fibrep}
E\to \zH_t(-,E). 
\end{equation}
The case when $t$ comes from a $cd$-structure was studied in depth by Voevodsky, who obtained several key results in his fundamental article \cite{htscd}. Using Voevodsky's results, one can show that if
$t$ comes from a $cd$-structure which meets some technical conditions \cite{chsw}*{Theorem 3.4}, which are satisfied in all the examples we consider here, then $E$ satisfies $t$-descent if and only if the map \eqref{map:fibrep} is a \emph{global weak equivalence}; this means that 
\[
E_*(X)\to \zH^{-*}_t(X,E)
\] 
is an isomorphism for all $X$. 

A cohomology theory of schemes $E$ which satisfies Zariski descent is determined by its value on affine schemes, and 
\begin{equation}\label{eq:affine}
E(R):=E(\Spec R)
\end{equation}
is a homology theory of commutative rings which satisfies Zariski descent. Each of the homology theories of rings considered in the previous sections extends to a cohomology theory of schemes which satisfies Zariski descent so that \eqref{eq:affine} holds. Moreover, we have the following landmark results of Thomason. Recall that a closed immersion is \emph{regular} if it is locally defined by a regular sequence. 
A closed subscheme $Y\subset X$ is of \emph{pure codimension} when all its irreducible components have the same codimension. The square 
\eqref{diag:square} is a \emph{regular blow-up} if $i$ is a regular closed immersion of pure codimension and $p$ is the blow-up along $i$.

\begin{theorem}\label{thm:bt} (Thomason, \cite{tt}*{Theorem 10.8}, \cite{teclate}*{Th\'eor\`eme 2.1})
$K$-theory of schemes satisfies Nisnevich descent and has the Mayer-Vietoris property for regular blow-up squares.
\end{theorem}

Thomason defined the $K$-theory of a quasi-compact and separated scheme $X$ in terms of the complicial category of perfect complexes of quasi-coherent sheaves; the main technical tool for proving Theorem \ref{thm:bt} is the Thomason-Waldhausen localization theorem (\cite{tt}*{Theorems 1.8.2 and 1.9.8}) which roughly says that $K$-theory maps sequences of complicial categories which induce exact sequences of derived categories into homotopy fibration sequences of spectra. Weibel introduced cyclic homology of schemes as the hyperhomology of the sheafified Connes' complex \cite{whcsch}. He and Geller showed in \cite{gwet}*{Theorem 4.8} that $HC$ satisfies Nisnevich descent. Keller proved a version of the Thomason-Waldhausen localization theorem for cyclic homology \cite{keloco}*{Theorem 2.4} and showed \cite{kesch}*{Section 5.2} that for quasi-compact separated schemes the cyclic homology of schemes introduced by Weibel in \cite{whcsch} admits a categorical description analogous to that of $K$-theory. The proof given in \cite{chsw} of the following theorem uses the latter result of Keller and the argument of Thomason's proof of Theorem \ref{thm:bt}.

\begin{theorem}\label{thm:reghc}(\cite{gwet}*{Theorem 4.8}, \cite{chsw}*{Theorems 2.9 and 2.10})
Cyclic, negative cyclic and periodic cyclic homology satisfy Nisnevich descent and have the Mayer-Vietoris property for regular blow-ups.
\end{theorem}

\begin{corollary}\label{coro:smooth}
Let $E\in\{HC, HN, HP, K\}$ and let $X\in\SchF$. Assume that $X$ is smooth. 
Then the map $E_*(X)\to H_{\cdh}^{-*}(X,E)$ is an isomorphism. 
\end{corollary}
\begin{proof} By \cite{chsw}*{Corollary 3.9}, any cohomology theory on $\SchF$ which satisfies Nisnevich descent and has the Mayer-Vietoris property for regular blow-ups also satisfies this.
\end{proof}
  
A homology theory of $k$-schemes $E$ is \emph{invariant under infinitesimal thickenings} if $E_*(X)\to E_*(X_\red)$ is an isomorphism for all $X$. An abstract blow-up square is a \emph{finite blow-up} if $p$ is a finite morphism. If $E$ satisfies Zariski descent then $E$ is invariant under infinitesimal thickenings (resp. has the Mayer-Vietoris property for finite blow-ups) if and only if the associated homology of commutative algebras \eqref{eq:affine} is nilinvariant (resp. satisfies excision).  

The Chern character \eqref{map:ch} extends to schemes. If $X\in\SchF$ and $k$ is of characteristic zero it is given by a map of spectra
$ch:K(X)\to HN(X)$; infinitesimal $K$-theory is defined as $K^{\inf}(X)=\hofi(K(X)\to HN(X))$. 

The following is an immediate corollary of Theorems \ref{thm:kinf}, \ref{thm:bt} and \ref{thm:reghc}.

\begin{corollary}\label{coro:kinfb}(\cite{chsw}*{Theorems 4.3 and 4.4})
Let $k$ be a field of characteristic zero. Then the infinitesimal $K$-theory of $k$-schemes is invariant under infinitesimal thickenings, satifies Nisnevich descent, and has the Mayer-Vietoris property for regular blow-ups and for finite blow-ups.
\end{corollary}

\section{Haesemeyer's theorem}\label{sec:chh} 
In this section $k$ is a field of characteristic zero. If $E$ is a cohomology theory of $k$-schemes, we write $\cF^E(X)=\hofi(E(X)\to H_{cdh}(X,E))$. Thus we have
a homotopy fibration sequence
\begin{equation}\label{seq:fibcdh}
\cF^E(X)\to E(X)\to \zH_{\cdh}(X,E).
\end{equation}
Both $\zH_{\cdh}(X,-)$ and $\cF^{-}(X)$ preserve homotopy fibration sequences; this fact is used in diagram \ref{diag:fk} below. 

\begin{theorem}\label{thm:chh}(Haesemeyer, \cite{chh}*{Theorem 6.4})
Let $k$ be a field of characteristic zero and $X\in\SchF$. Then  $KH(X)\weq \zH_\cdh(X,K)$.
\end{theorem}

It follows from the theorem above that for $E=K$ and $X=\Spec A$, the sequence \eqref{seq:fibcdh} is equivalent to the middle row of diagram \eqref{diag:knil}. For $X\in\SchF$ we have a homotopy commutative diagram whose rows and columns are homotopy fibration sequences:
\begin{equation}\label{diag:fk}
\xymatrix{
\cF^{K^{\inf}}(X)\ar[r]\ar[d]&K^{\inf}(X)\ar[d]\ar[r]& H_{\cdh}(X, K^{\inf})\ar[d]\\
\cF^K(X)\ar[r]\ar[d]&K(X)\ar[d]\ar[r]&KH(X)\ar[d]\\
\cF^{HN}(X)\ar[r]&HN(X)\ar[r]&H_{\cdh}(X,HN). 
}
\end{equation}

Theorem \ref{thm:chh} implies that $KH$ satisfies $cdh$-descent. On the other hand we know from \cite{kh}*{Theorem 1.2} that it is also \emph{homotopy invariant}, that is, it sends the projection $X\times \zA^1\to X$ to an isomorphism $KH_*(X)\cong KH_*(X\times\zA^1)$. For the proof of Theorem \ref{thm:chh}, Haesemayer first showed that $KH$ has the Mayer-Vietoris property for regular
blow-ups (\cite{chh}*{Theorem 3.6}). In almost all the rest of the proof, the only other properties of $KH$ that are used are excision and invariance under infinitesimal thickenings. Homotopy invariance is used only once, in \cite{chh}*{Proposition 3.7}. It was found later that homotopy invariance is not needed; this led to the following generalization
of Haesemeyer's theorem.

\begin{theorem}\label{thm:chhg}(\cite{chsw}*{Theorem 3.12})
Let $k$ be a field of characteristic zero and let $E$ be a cohomology theory of $k$-schemes. 
Assume that $E$ satisfies excision, is invariant under infinitesimal thickenings, satisfies Nisnevich descent, and has the
Mayer-Vietoris property for regular blow-ups. Then $E$ satisfies $\cdh$-descent.
\end{theorem}

By Corollary \ref{coro:kinfb}, Theorem \ref{thm:chhg} applies to $K^{\inf}$. It follows that $\cF_*^{K^{\inf}}(X)=0$ in diagram \eqref{diag:fk}, and therefore we have a weak equivalence
\[
\cF^K(X)\weq \cF^{HN}(X).
\]
By Cuntz-Quillen's theorem \eqref{cq} and by Theorems \ref{thm:goo} and \ref{thm:reghc}, $HP$ also satisfies the hypothesis of Theorem \ref{thm:chhg}. It follows from this and the $SBI$-sequence that there is a weak equivalence 
\[
\cF^{HC}(X)[-1]\weq\cF^{HN}(X).
\]
Summing up, we have proved the following.

\begin{theorem}\label{thm:fhc}(\cite{chw}*{Theorem 1.6}, see also \cite{chsw}*{Corollary 3.13 and Theorem 4.6})
Let $k$ be a field of characteristic zero and $X\in\SchF$. Then there is a homotopy fibration sequence
\[
\cF^{HC}(X)[-1]\to K(X)\to KH(X).
\]
\end{theorem}

Theorems \ref{thm:chh} and \ref{thm:fhc} together say that the obstruction to $\cdh$-descent in algebraic $K$-theory
is measured by $\cF^{HC}$. By Corollary \ref{coro:smooth} and Hironaka's desingularization theorem \cite{hiro} the groups $\cF_*^{HC}(X)$
can be computed in terms of the cyclic homology of $X$ and of the cyclic homology of a finite array (called a hyperresolution) of smooth schemes that appear in the desingularization process. In this sense we can say that the fiber of $K(X)\to KH(X)$ can be computed in terms of cyclic homology. Note that this represents a lot of progress from our starting point. Indeed regarding $KH$ as $K$-theory made homotopic lead us to a map $K_*^{\nil}\to HC_{*-1}$ (see \eqref{diag:knil})
that is never an isomorphism in the commutative case (see Remark \ref{rem:nofunca}); regarding it instead as a version of $K$-theory with $\cdh$-descent lead us to a character $K_*^{\nil}=\cF_*^K\to \cF_{*-1}^{HC}$ which is always an isomorphism. Haesemeyer's theorem made all the difference. 

Theorems \ref{thm:chh} and \ref{thm:fhc} are powerful tools for studying the $K$-theory of singular schemes. In the next sections we review a number of results whose proofs used these tools.

\begin{remark}
Building on work of Ayoub and ideas of Voevodsky, Cisinski proved in \cite{cis}*{Th\'eor\`eme 3.9} that $KH$ satisfies $cdh$-descent in the category of Noetherian schemes of finite dimension. Since the latter category includes the schemes of finite type over a field of any characteristic, Cisinki's result is far more general than Haesemeyer's Theorem \ref{thm:chh}. However, Cisinski's proof relies heavily on homotopy invariance, and therefore no analog of Theorem \ref{thm:chhg} can be derived from his argument. A version of Theorem \ref{thm:chhg} for schemes of finite type over a perfect field which admits strong resolution of singularities, due to Geisser-Hesselholt, is given in Theorem \ref{thm:ghchh} below. 
\end{remark}

\section{Weibel's dimension conjecture}\label{sec:dimconj}

\begin{conjecture}\label{conj:dim} (Weibel's dimension conjecture \cite{wana}*{Questions 2.9})
Let $X$ be a Noe- therian scheme of dimension $d$. Then $K_m(X)=0$ for $m<-d$ and $X$ is $K_{-d}$-regular.
\end{conjecture}

The $KH$-version of the conjecture for schemes over a field of characteristic zero was settled by Haesemeyer:

\begin{theorem}\label{thm:chhdim} (Haesemeyer, \cite{chh}*{Theorem 7.1})
Let $k$ be a field of characteristic zero and let $X\in\SchF$. Then $KH_m(X)=0$ for $m<-\dim X$. 
\end{theorem}

\begin{proof} By Theorem \ref{thm:chh} and \cite{chh}*{Theorem 2.8} there is a spectral sequence
\[
E_2^{p,q}=H_\cdh^p(X,aK_{-q})\Rightarrow KH_{-p-q}(X).
\]
The $E_2$-term is the cohomology of the  $\cdh$ sheaf associated to the $K$-theory groups. We have $aK_{-q}=0$ for $q>0$ by Hironaka's desingularization theorem and the fact that negative $K$-theory vanishes on smooth schemes. By a result of Suslin-Voevodsky \cite{svbk}*{Theorem 5.13}, 
if $n>d=\dim X$ then $H^n_\cdh(X,\cS)=0$ for every $\cdh$ sheaf $\cS$; thus $E^{p,q}_2=0$ if either $p>d$ or $q>0$. The theorem is immediate from this. 
\end{proof}

\begin{theorem}\label{thm:dim} (\cite{chsw}*{Theorem 6.2})
Weibel's conjecture holds for schemes essentially of finite type over a field of characteristic zero. 
\end{theorem}

The vanishing part of Theorem \ref{thm:dim} can be proved using Theorems \ref{thm:fhc} and \ref{thm:chhdim} and a similar spectral sequence argument; the regularity part requires more work (see \cite{chsw}*{Section 6}).

\section{Singularity and the obstruction to \textup{K}-regularity}\label{sec:reg}

Let $F:\rings\to\ab$ be a functor from rings to abelian groups. If $A$ is a ring, then the inclusion $A\subset A[t]$ is split by evaluation at zero. Hence for
\[
NF(A)=\coker(F(A)\to F(A[t]))
\]
we have a direct sum decomposition 
\[
F(A[t])=F(A)\oplus NF(A).
\]
Similarly,
\[
F(A[t_1,t_2])=F(A)\oplus NF(A)\oplus NF(A)\oplus N^2F(A),
\]
where $N^2F=N(NF)$. Iterating this process one obtains an expression for $F(A[t_1,\dots,t_n])$
in terms of $F(A)$ and of the groups $N^jF(A)$ for $j\ge 1$. Hence $A$ is $F$-regular if and only if $N^jF(A)=0$ for all $j\ge 1$. 

Kassel proved (\cite{kas}*{Example 3.3}) that if $A$ is a unital associative $\zQ$-algebra, then
\begin{equation}\label{nhc}
NHC_n(A)=HH_n(A)\otimes t\zQ[t].
\end{equation}
Here $HH$ is \emph{Hochschild homology}; it is related to cyclic homology by Connes' $SBI$-sequence
\[
HC_{n-1}(A)\overset{B}\to HH_n(A)\overset{I}\to HC_n(A)\overset{S}\to HC_{n-2}(A).
\]
For example if $A$ is commutative, then $HH_0(A)=A$ and
\begin{equation}\label{hh1}
HH_1(A)=\Omega^1_A
\end{equation}
is the module of absolute \emph{K\"ahler $1$-differential forms}. 
Using \eqref{nhc}, \eqref{hh1}, and the K\"unneth formula for Hochschild homology \cite{chubu}*{Proposition 9.4.1}, we obtain
\begin{align}
N^2HC_n(A)=&NHH_n(A)\otimes t\zQ[t]\label{nhh}\\
=&HH_n(A)\otimes\zQ[t]\otimes\zQ[t]\oplus HH_{n-1}(A)\otimes\Omega^1_{\zQ[t]}\otimes\zQ[t].\nonumber
\end{align}
Iterating this process one obtains a formula for $N^jHC_n(A)$ for all $j$ in terms of $HH_{n-p}(A)$ $(p\le j)$ (see \cite{chwnk}*{Formula 1.5}). 
Like cyclic homology, Hochschild homology is defined for schemes (see \cite{whcsch}) and has all the properties stated for cyclic homology in Theorem \ref{thm:reghc} and Corollary \ref{coro:smooth} (\cite{chsw}*{Theorem 2.9}). Formulas \eqref{nhc} and \eqref{nhh} generalize to schemes. Moreover if $k$ is a field of characteristic zero and $X\in\SchF$, then for $V=t\zQ[t]$ and $dV=\Omega^1_{\zQ[t]}$, we have (\cite{chwnk}*{Lemma 3.2 and Corollary 3.3}
\begin{gather*}
N\cF_n^{HC}(X)=\cF_n^{HH}(X)\otimes V,\\
N^2\cF^{HC}_n(X)= \bigl(\cF_n^{HH}(X)\otimes V\otimes V\bigr)
\oplus \bigl(\cF_{n-1}^{HH}(X)\otimes V\otimes dV\bigr).
\end{gather*}
In view of Theorem \ref{thm:fhc} and of the fact that $KH$ is homotopy invariant, we obtain the following formulas:
\begin{align*}
NK_n(X)=&\cF_{n-1}^{HH}(X)\otimes V\\
N^2K_n(X)=&\bigl(\cF_{n-1}^{HH}(X)\otimes V\otimes V\bigr)
\oplus \bigl(\cF_{n-2}^{HH}(X)\otimes V\otimes dV\bigr). 
\end{align*}
Of course we can iterate this to produce formulas for $N^jK_n(X)$ in terms of $\cF^{HH}_{n-p}(X)$
($p\le j-1$). Summing up, we obtain the following.

\begin{proposition}\label{prop:obsreg}
Let $k$ be a field of characteristic zero, let $X\in\SchF$ and let $n\in\zZ$. Then $X$ is $K_n$-regular if and only if $\cF_{m}^{HH}(X)=0$ for all $m\le n-1$.
\end{proposition}

In view of \eqref{seq:fibcdh}, the equivalent conditions of the proposition are also equivalent to the assertion that the map
\[
HH_m(X)\to \zH^{-m}_{\cdh}(X,HH)
\]
is an isomorphism for $m\le n-1$ and a surjection for $m=n$. By \cite{chw}*{Proposition 2.1 and Theorem 2.2}, $\zH^{*}_{\cdh}(X,HH)$ decomposes as follows:
\[
\zH^{-m}_{\cdh}(X,HH)=\bigoplus_{i\ge 0}H^{i-m}_{\cdh}(X,a\Omega^i).
\]
Here $H^*_{\cdh}(X,a\Omega^i)$ is the cohomology of the $\cdh$-sheaf associated to K\"ahler $i$-differential forms. Summing up, $K$-regularity questions translate into comparing Hochschild 
homology with $\cdh$ cohomology of differential forms. This point of view has been useful for solving the following old questions about $K$-regularity.

\begin{conjecture}\label{conj:vorst} (Vorst's regularity conjecture \cite{vorst})
Let $R$ be a commutative unital ring of dimension $d$, essentially of finite type over a field $k$. If $R$ is $K_{d+1}$-regular
then $R$ is a regular ring.
\end{conjecture}

For $d=0$, the conjecture is trivial; the case $d=1$ was proved by Vorst \cite{vorst}*{Theorem 3.6}. Recall that a commutative unital ring $R$ is \emph{regular in codimension $d$} if the local ring $R_\fp$ is regular for each prime $\fp$ of codimension $d$. A Noetherian ring of dimension $d$ is regular if and only if it is regular in codimension $d$. 

\begin{theorem}\label{thm:chwvorst}(\cite{chw}*{Theorem 0.1})
Let $R$ be a commutative unital ring which is essentially of finite type over a field $k$ of characteristic zero. Assume that $R$ is $K_n$-regular. Then $R$ is regular in codimension $<n$. In particular, if $R$ is $K_{\dim R+1}$-regular, then it is regular.
\end{theorem}

\begin{question}(Bass' question \cite{bass}*{Question $(VI)_n$})

Does $NK_n(R)=0$ imply $N^2K_n(R)=0$?
\end{question}

\begin{theorem}\label{thm:bass}(\cite{chwnk}*{Corollary 6.7},\cite{chwwq}*{Theorem 4.1})

\smallskip

a) For any field $F$ algebraic over $\zQ$, the 2-dimensional normal algebra
$$R=F[x,y,z]/(z^2+y^3+x^{10}+x^7y)$$ 
has $NK_0(R) = 0$ but $N^2K_0(R) \neq 0.$

b) Suppose $R$ is essentially of finite type over a field 
of infinite transcendence degree over $\zQ$. Then $NK_n(R) = 0$ implies
that $R$ is $K_n$-regular.
\end{theorem}

The proof in \emph{loc.cit.} of the theorem above employs the method described in the paragraph after Proposition \ref{prop:obsreg}. Gubeladze gave a different proof of part b) in \cite{gubelb}*{Theorem 1}. 

The next conjecture concerns abelian monoids; we shall use multiplicative notation. An abelian monoid is called \emph{cancellative} if $ac=bc$ implies $a=b$, and \emph{torsion-free} if $a^n=b^n$ with $n\in\zZ_{\ge 1}$ implies $a=b$. An element $u\in M$ is called a \emph{unit} if it has an inverse in $M$. If $k$ is a commutative ring and $M$ a commutative monoid,
then the \emph{monoid $k$-algebra} $k[M]$ is the set of finitely supported functions $M\to k$ equipped with pointwise addition and convolution product. Any element of $k[M]$ can be written uniquely as a finite $k$-linear combination $\sum_a\lambda_a\chi_a$ where $\chi_a$ is the characteristic function $\chi_a(b)=\delta_{a,b}$. Each integer $c\ge 2$ defines a $k$-algebra homomorphism
\[
\theta_c:k[M]\to k[M],\qquad \theta_c(\chi_a)=\chi_{a^c}.
\]
The map $\theta_c$ is called the \emph{dilation} of ratio $c$. If $F:\ring\to\ab$ is a functor and $\fc=(c_1,c_2,\dots)$ is a sequence of integers $c_i\ge 2$, then we have an inductive system $\{\theta_{c_n}:F(k[M])\to F(k[M]):n\ge 1\}$. We write
\[
F(k[M])^{\fc}=\colim_{\theta_{c_n}} F(k[M])
\]
for its colimit. 

\begin{conjecture}\label{conj:gubel}(Gubeladze's nilpotence conjecture, \cite{gubelc}*{Conjecture 2.1})
Let $k$ be a commutative ring, $M$ an abelian monoid and $\fc=(c_1,c_2,\dots)$ a sequence of integers $c_i\ge 2$. Assume that $k$ is regular 
Noetherian and that $M$ is cancellative and torsion-free and has no non-trivial units. Then 
\[
K_*(k[M])^{\fc}=K_*(k).
\]
\end{conjecture}

\begin{remark} 
Under the conditions of the conjecture, the map $k[M]\to k$ is a polynomial homotopy equivalence (see the proof of \cite{chwwp}*{Corollary 8.4}). Thus Gubeladze's conjecture is a homotopy invariance statement.
\end{remark}

\begin{theorem}\label{thm:gubel0}(Gubeladze, \cite{gubelinv}*{Theorem 1.2},\cite{gubelcoeff}*{Theorem 1})
Conjecture \ref{conj:gubel} holds if $k$ contains a field of characteristic zero.
\end{theorem}

Gubeladze's proof uses relies on Theorem \ref{thm:kabi} and on his previous work on special cases of the conjecture. A different
proof, using Theorem \ref{thm:chhg}, was given in \cite{chwwt}*{Corollary 6.10} for the particular case when $k$ is a field. The general case when $k$ is a regular ring containing a field (of any characteristic) was proved in \cite{chwwp}*{Theorem 0.2}; this is discussed in Section \ref{sec:p}.

\section{A glimpse at characteristic \topdf{$p>0$}{p}.}\label{sec:p}

We have seen in Section \ref{sec:obse} that the obstructions to nilinvariance and excision with rational coefficients are computed by the Chern character to negative cyclic homology. To compute these obstructions in the $p$-adically complete case for a prime $p$, one has to replace the Chern character by the \emph{cyclotomic trace} of B\"okstedt-Hsiang-Madsen \cite{bhm}
\[
\tr:K(R)\to TC(R),
\]
 which takes values in \emph{topological cyclic homology}. 
 
\begin{theorem}\label{thm:mc} (McCarthy, \cite{macnil}*{Main Theorem})
Let $R$ be a unital ring, $I\triqui R$ a nilpotent ideal and $p>0$ a prime. Then $\tr:K(R:I)\to TC(R:I)$ becomes a weak equivalence after $p$-completion.
\end{theorem}

The topological cyclic homology spectrum $TC$ is the homotopy limit of a pro-spectrum $TC^n$; the following theorem is formulated in terms of the latter pro-spectrum. It implies that 
$K(R,S:I)\to TC(R,S:I)$ becomes a weak equivalence after $p$-completion.

\begin{theorem}\label{thm:ghkabi} (Geisser-Hesselholt \cite{ghkabi}*{Theorem 1})
Let $f:R\to S$ be a homomorphism of unital associative rings, let $I\triqui R$
be a two-sided ideal and assume that $f:I\to f(I)$ is an isomorphism onto
a two-sided ideal of $S$. Then the map induced by the cyclotomic trace map
\[
K_q(R, S:I, \zZ/p^\nu)\to TC^n_q(R,S:I,\zZ/p^\nu)
\]
is an isomorphism of pro-abelian groups, for all integers $q$, all primes $p$,
and all positive integers $\nu$.
\end{theorem}

\begin{remark}\label{rem:dunk}
In their recent article \cite{dunk}, Dundas and Kittang have shown that Goodwillie's global cyclotomic trace $K\to \TC$ induces an integral isomorphism $K_*(R,S:I)\to \TC_*(R,S:I)$.  
\end{remark}

The following is a version of Theorem \ref{thm:chhg} for perfect fields which admit \emph{strong resolution of singularities}. This means that for
every integral scheme $X$ separated and of finite type over $k$, there
exists a sequence of blow-ups
$$X_r \to X_{r-1} \to \dots \to X_1 \to X_0 = X$$
such that the reduced scheme $X_r^{\red}$ is smooth over
$k$; the center $Y_i$ of the blow-up $X_{i+1} \to X_i$ is connected
and smooth over $k$; the closed embedding of $Y_i$ in $X_i$ is
normally flat; and $Y_i$ is nowhere dense in $X_i$.  

\begin{theorem}\label{thm:ghchh} (Geisser-Hesselholt, \cite{ghneg}*{Theorem 1.1})
Let $k$ be an
infinite perfect field such that strong resolution of singularities
holds over $k$, and let $\{ F^n(-) \}$ be a presheaf of pro-spectra on
the category of schemes essentially of finite type over
$k$. Assume that $\{ F^n(-) \}$ takes infinitesimal thickenings to weak
equivalences and finite abstract blow-up squares to homotopy cartesian
squares. Assume further that each $F^n(-)$ takes elementary Nisnevich
squares and squares associated with blow-ups along regular embeddings
to homotopy cartesian squares. Then the canonical map defines a weak
equivalence of pro-spectra 
$$\{ F^n(X) \} \xrightarrow{\sim}
\{ \mathbb{H}_{\cdh}^{\boldsymbol{\cdot}}(X,F^n(-)) \}$$
for every scheme $X$ essentially of finite type over $k$.
\end{theorem}

By Theorems \ref{thm:kh}, \ref{thm:mc}, and \ref{thm:ghkabi}, the theorem above applies to $KH$ and to the fiber of the cyclotomic trace. 
Using this, Geisser and Hesselholt obtained the following result about Weibel's dimension 
conjecture.

\begin{theorem}\label{thm:ghneg}(Geisser-Hesselholt, \cite{ghneg}*{Theorem A})
Let $k$ be an infinite perfect field of
characteristic $p > 0$ such that strong resolution of singularities
holds over $k$, and let $X$ be a $d$-dimensional scheme essentially of
finite type over $k$. Then $K_q(X)$ vanishes for $q < -d$.
\end{theorem}

Geisser and Hesselholt also obtained the following result about Vorst's conjecture. 

\begin{theorem}\label{thm:ghvorst}(Geisser-Hesselholt, \cite{ghvorst}*{Theorem A})
Let $k$ be an infinite perfect field of
characteristic $p > 0$ such that strong resolution of singularities
holds over $k$. Let $R$ be a localization of a $d$-dimensional
commutative $k$-algebra of finite type and suppose that $R$ is
$K_{d+1}$-regular. Then $R$ is a regular ring.
\end{theorem}

If we restrict our attention to toric varieties, the assumption that the ground field admits strong resolution
of singularities can be dropped. The Bierstone-Milman theorem \cite{BM}*{Thm.\,1.1} provides a chacteristic-free resolution
of singularities for such varieties that is sufficient to prove the following toric version of Haesemeyer's theorem.

\begin{theorem}\label{thm:chhtoric}(\cite{monoidres}*{Theorem 1.1})
Assume $k$ is an infinite field 
and let $\cG$ be a cohomology theory on $\SchF$. If $\cG$ satisfies the Mayer-Vietoris property
for Zariski covers, finite abstract blow-up squares, and blow-ups along
regularly embedded closed subschemes, then $\cG$ satisfies the Mayer-Vietoris
property for all abstract blow-up squares of toric $k$-varieties.
\end{theorem}

Theorems \ref{thm:ghkabi} and \ref{thm:chhtoric} were used to prove the following.

\begin{theorem}\label{thm:gubelp}(\cite{chwwp}*{Theorem 0.2})
Gubeladze's conjecture \ref{conj:gubel} holds if $k$ contains a field of characteristic $p>0$.
\end{theorem}

As we saw in Theorem \ref{thm:gubel0}, the case when $k\supset \zQ$ of the theorem above had already been proved by Gubeladze in \cite{gubelcoeff}. A different proof of Gubeladze's theorem, using Theorem \ref{thm:chhtoric}, was also given in \cite{chwwp}*{Theorem 7.5}.

\end{document}